\documentclass[12pt]{article}

\usepackage[dvips]{graphics}
\usepackage{amsmath,amsfonts,amssymb,amsthm}
\usepackage{epsfig}

\newcommand{\dt}{\Delta t}

\newtheorem{thrm}{Theorem}

\newcommand{\mj}{\mathcal{J}^\infty}

\def\N{\mathbb N}

\def\R{\mathbb R}
\def\P{\mathbb P}

\begin{document}

\title{Numerical approximation of a linear elasticity model}
\author{Samir Kumar Bhowmik\\
Department of Mathematics, University of Dhaka, Dhaka 1000, Bangladesh\\
Bhowmiksk@gmail.com}

\maketitle

\begin{abstract}
We consider a nonlocal linear elastic wave model. We approximate the model using a spectral Galerkin method in space and analyze error in  such an approximation. We perform  some numerical experiments to demonstrate the scheme.
\end{abstract}

\textbf{keywords:} convolution integral; polynomial approximation; error estimate.


\section{Introduction}
Integro-differential equations
 arise while modeling many problems in physics, biology and other areas which involve non-local diffusion/dispersal mechanisms~\cite{C.Fife, Dug, EEOW01}.
A nonlocal elastic wave model can be written as  \cite{EEOW01}
\begin{equation}\label{eq:001aa}
\frac{\partial^2}{\partial t^2} u(x,t) = \rho \mathcal{L}u(x, t) + g(x,t)\quad (t, x)\in \Omega\times (0, T),
\end{equation}
where
$$
   \mathcal{L}\phi(x)=\int_{\Omega} \mj(x-y)\phi(y)dy + \alpha(x)\phi(x),
$$
$(0, T)$, $T>0$ is the time interval under consideration, $\alpha(x)$ is a time-independent coefficient and $g(x, t)$ is an inhomogeneity. Here we consider $u(x, 0)=u_0$, and $\frac{\partial}{\partial t}u(x, 0)=v_0$, $\Omega\subset\R$, where $u(x, t)$  denotes the displacement field and $\partial_t u(x, t)$ represents velocity.
Here the $L^1(\R)$ kernel $\mj (x)$ satisfies  $\mj(x)=\mj(-x)$ (symmetry).

We assume that   $\alpha(x)$ is a constant given by
$
\alpha=-\int_{\Omega} \mj(x) dx,
$
 and for simplicity we consider a normalized kernel function ($\int \mj(x)dx =1$).
Then  \eqref{eq:001aa} can be written as
\begin{equation}\label{eq:001}
\frac{\partial^2}{\partial t^2} u(x,t) = \rho \mathcal{L}u(x, t) + g(x,t)
\end{equation}
where
\[
\mathcal{L}\phi(x)=\int_{\Omega} \mj(x-y)(\phi(y)-\phi(x)) dy.
\]
Here we restrict ourself by $\mj \ge 0$ (nonnegative) as well.

Now let $y-x=z$, then $dy=dz$. Considering $\Omega = \mathbb{R}$ \cite{SKB02}  in (\ref{eq:001}) we have
\[
I_{1}=\varepsilon \int_{\mathbb{R}}\mj(x-y)\left(u(y,t)-u(x,t)\right)dy
=\varepsilon \int_{\mathbb{R}} \mj(z)\left(u(x+z,t)-u(x,t)\right)dz.
\]
Expanding $u(x+z)$ in terms of Taylor series around $z=0$ we have
\[
I_1=\varepsilon\int_{\mathbb{R}} \mj(z)\left(u(x,t)+z u_{x}(x,t)+\frac{z^2}{2!} u_{xx}(x,t)
                                                          +\cdots\cdots -u(x,t)\right)dz.
\]
Thus we have
\begin{equation}\label{tx:f}
u_{tt}=\varepsilon u_x \int_{\mathbb{R}} z \mj(z)dz+\frac{\varepsilon u_{xx}}{2!}\int_{\mathbb{R}} z^2 \mj(z)dz +\cdots\cdots + g(x, t).
\end{equation}
As $J(z)$ is symmetric, integrals with odd powers of $z$ are zero.
So (\ref{tx:f}) takes the form
$
u_{tt}=C_2 u_{xx}+C_4 u_{xxxx}+...+g(x, t).
$
where $C_{2 m}=\frac{\varepsilon}{2m!}\int_{\mathbb{R}} z^{2m}\mj(z)dz$; for each $ m=1, 2, 3,\cdots.$
Ignoring the higher order terms we have
\begin{equation}\label{pde1:f}
u_{tt}=C_2 u_{xx}+g(x, t).
\end{equation}
The above equation is the familiar non-homogeneous wave equation, and is a first order approximation of the non-local equation \eqref{eq:001}. A comparison between the operators $u_{xx}$ and $\mathcal{L}u$  is well presented in \cite{SKB02, JCAC2004}.
%
Now following \cite{SKB02} it is well understood that
 $\|\mathcal{L}\|$ is bounded and $\mathcal{L}$ is negative semi-definite,  if the kernel function $\mj$ is considered non-negative, symmetric and normalized.
These properties are important for the analysis we performed in this study.

Emmrich et. al.~\cite{EEOW01} consider the non-local elastic model \eqref{eq:001}. They analyze mathematical and numerical solutions of the model. They compare the non-local model with that of a local continuum model. They claim two advantages of using this non-local model.  They propose and implement  a quadrature based numerical scheme considering infinite spatial domain.

It is well understood from \cite{SamirKumarBhowmik04, D.J.Duffy03} that an integro-differential
equation of type \eqref{eq:001} defined in the infinite
domain can also be defined in a truncated finite domain $[A, B]$ where $A$ and $B$ depend  on the decay of the kernel function $J^\infty(x)$. It is to note that a closed form formula to find suitable
$A$ and $B$ is well
presented in \cite{D.J.Duffy03}. Thus the analysis and the approximation
in a periodic $\Omega = [0, 1]$ spatial domain can also be applied in any bounded interval $[A, B]$ (and vice-versa).
Considering the kernel of the convolution integral as
$
g_{\delta}(x) = \sqrt{\frac{1}{2\pi \delta}} \exp \left(-\frac{y^2}{2\delta^2}\right),
$
the author in \cite{D.J.Duffy03}   formulate
$
 A = +\sqrt{-2\delta^2 \log (\delta \varepsilon \sqrt{2\pi})},
$
and
$
B= -A,
$
where $\varepsilon>0$ considered so that $g_{\delta}(x) \ge \varepsilon$.

One can consider the model in a spatial periodic domain~\cite{SKB02} as well.
If we consider a spatially one-periodic initial function $u(x, 0)$,
then for all $x\in \R$ and $t \in \R_+$
$
u(x, t) = u(x+L, t).
$
Then (\ref{eq:001}) can be written as
\begin{eqnarray}\label{f:equ01a}
 u_{tt} &=& (\mathcal{L}u(\cdot, t))(x) +g(x,t)
        = \int_{\Omega} \mathcal{J}(x-y)\left(u(y, t)-u(x, t)\right)dy +g(x,t)
\end{eqnarray}
where
$
  \mathcal{J}(x) = \sum_{r=-\infty}^{\infty} \mj (x-rL)
$
for all $x \in [0, L]$, $L>0$. We are interested to consider both periodic and non-periodic domain for spatial approximations of the model.

Convolution models have been extensively investigated both theoretically and numerically~\cite{SKB_pap_linearide, EEOW01}, but a lot more yet to be done to get a highly accurate solver. Since an analytic solution can not be obtained  always, an efficient and highly accurate numerical scheme has to be developed. Spatial approximation of this type of convolution models are interesting as well as challenging.
The unknown function under integral sign and the nonlinearity involved in the model make the
approximation more challenging. There are various ways to handle such problems.
In most cases scientists use some lower order schemes (with midpoint quadrature rules for integration \cite{SKB02, SKB_pap_phase_trans}) to serve their purpose. Thus there is much room for improvement and we find an interest of
presenting and analyzing a higher order technique for space integration. To the best of our knowledge, Legendre spectral Galerkin method has not been performed for this model problem yet and we find an interest  to investigate such  a scheme.

In this article, we start with numerical accuracy analysis of a spectral galerkin approximation in section 2 followed by numerical implementations of the scheme in section 3. We present several computer generated solutions with relevent discussion in section 4. We implement the schemes in MATLAB.
%
%
\section{Accuracy of spectral Galerkin  approximation}
%
%
In this section, we look for spectral Galerkin scheme for space approximation and investigate convergence of such a scheme. Here we consider Legendre spectral polynomials. It is our goal to find $u=u(x, t)$, $x\in \Omega$ such that
\[
u_{tt}(x, t)=\mathcal{L} u (x, t)+g(x, t), \forall x\in\Omega,
\]
and we seek for weak form to find $u\in L_2(\Omega)$ such that
\begin{equation}\label{f:specweak01}
 (u_{tt}, v) =  (\mathcal{L} u (x, t), v) + (g(x, t), v), \quad \forall \quad v\in L_2(\Omega).
\end{equation}
We denote $\N$ be the set of all nonnegative integers. For any $N \in \N$, $\P_N$ denotes the set of all algebraic polynomials of degree at most  $N$ in $\Omega$. We define Legendre spectral Galerkin approximation  of (\ref{eq:001}) as:
Find
\[
  u^N \in \P_N, \quad \text{such that}
\]
\begin{equation}\label{f:specweak01a}
 (u^N_{tt}, v) =  (\mathcal{L} u^N (x, t), v) + (g(x, t), v), \quad \forall \quad v\in \P_N(\Omega).
\end{equation}
We define projection operator $\P_N$ such that
\[
(u, v) =(P_N u, v ),  \quad \forall \quad v\in \P_N.
\]
Then \cite{P.Solin.K.Segeth.I.Dolezel2004, C.Canuto.M.Y.Hussaini.A.Quarterni.T.A.Zang2006}
\[
\|u-P_N u\| \le C N^{-m} |u|_{H^m(\Omega)} \le C N^{-m} \|u\|_{H^m(\Omega)}
\]
where
\[
\|u\|_{H^m(\Omega)} = \left( \sum_{k=0}^m \|u^{k}\|\right)^{1/2},
\]
and
\[
|u|_{H^m(\Omega)} = \left( \sum_{k=\min{m, N+1}}^m \|u^{k}\|\right)^{1/2},
\]
and
\[
\|u\|_{m, \infty} = \max_{0\le k\le m}\left\{ \|u^{k}\|_\infty\right\},
\]
and $C$ denotes a nonnegative constant which is independent of $N$.

Let $e^N=u-u^N$ be the error in spectral Galerkin approximation to the Legendre spectral Galerkin solution $u^N$ of (\ref{f:specweak01}). Now from (\ref{f:specweak01}) and (\ref{f:specweak01a}) we have
\begin{equation}\label{f:err01}
((u-u^N)_{tt}, v) = (\mathcal{L}(u-u^N), v).
\end{equation}
\begin{thrm}\label{f:thrm_galr}
If $u^N\in \P_N$  is a solution of  (\ref{f:specweak01a}), and $u \in S$ is a  solution of (\ref{f:specweak01}) then the following inequality holds
\[
\|(u(\cdot, t)-u^N(\cdot, t))_{tt}\| \le C N^{-m} \|u(\cdot, t)\|_{m}, 
\]
for some $C>0$.
\end{thrm}
We need the following results to prove Theorem~\ref{f:thrm_galr}.
\begin{thrm}\label{f:semidefn}~\cite{SKB_pap_phase_trans}
If $\mj(x) \ge 0$, $\mj(-x) = \mj(x)$ for all $x \in \Omega$  and $u(x)\in \mathbb{R}$ with $u \in L_2(\Omega)$
then $\mathbb{L}$ is negative semi-definite.
\end{thrm}
\begin{thrm}\label{f:boundL2}~\cite{SKB_pap_phase_trans}
If $\mj(x) \ge 0$ $\forall$ $x \in \mathbb{R}$,
$ \mj(x) \in L_2(\mathbb{R})$ and $\int_{\mathbb{R}} \mj(x) dx = 1,$ then for any $\Omega \subseteq \mathbb{R}$,
$\mathbb{L}$ is bounded and
$
\|\mathbb{L}\| \le 2.
$
\end{thrm}
\begin{proof}[Proof of Theorem~\ref{f:thrm_galr}]
We write
\[
  e^N=u-u^N = u - P_Nu + P_N u -u^N = \rho_N +\theta_N.
\]
Thus we rewrite (\ref{f:err01}) as
\begin{equation}\label{f:err01}
((\theta_N(\cdot, t))_{tt}, v) = (\mathcal{L} \rho^N(\cdot, t) , v) + (\mathcal{L} \theta^N(\cdot, t) , v).
\end{equation}
Replacing $v \in \P_N$ by $\theta_N(\cdot, t) \in \P_N$  we get
\[
|((\theta_N(\cdot, t))_{tt}, \theta_N(\cdot, t))| \le |(\mathcal{L}\rho_N(\cdot, t), \theta_N(\cdot, t))| + |(\mathcal{L}\theta_N(\cdot, t), \theta_N(\cdot, t))|
\]
and then applying theorem~\ref{f:semidefn} (negative definiteness of $\mathcal{L}$)
\[
|((\theta_N(\cdot, t))_{tt}, \theta_N(\cdot, t))| \le |(\mathcal{L}\rho_N(\cdot, t), \theta_N(\cdot, t))|.
\]
Thus applying theorem~\ref{f:boundL2}
\[
\|(\theta_N(\cdot, t))_{tt}\| \|\theta_N(\cdot, t)\| \le C \|\rho_N(\cdot, t)\| \|\theta_N(\cdot, t)\|
\]
and then canceling common terms and bounds we get
\[
\|(\theta_N(\cdot, t))_{tt}\| \le C N^{-m} \|u(\cdot, t)\|_{H^{m}}.
\]
\end{proof}
\section{Numerical implementation}
Here we implement spectral Galerkin schemes considering Legendre  spectral polynomials.
The Legendre polynomials satisfy three term recurrence relation:
\[
L_0(x) =1,
\]
\[
L_1(x) =x,
\]
and
\[
(n+1) L_{n+1} (x) = (2n+1)xL_{n(x)} - n L_{n-1}(x), \quad n\ge 1,
\]
with
\[
\int_{-1}^{1} L_k(x) L_j(x) dx = \frac{2}{2k+1}\delta_{kj},
\]
\[
\delta_{kj} = \left\{
\begin{array}{cc}
1& \text{if} \quad k=j\\
0& \text{if} \quad k\ne j.
\end{array}
\right.
\]

Here we consider $\Omega=[-1, 1]$ for numerical implementations and one may extend the scheme to any interval $[-L, L]$, $L>0$.
We define
\[
u^N(x, t) = \sum_{k=0}^{N} a_k(t) L_k(x).
\]
Replacing $v\in \P_N$ by $L_k\in \P_N$ in (\ref{f:specweak01}) we have
\begin{equation}\label{f:specweak02}
 (u^N_{tt}(x, t), L_k(x)) =  (\mathcal{L} u^N (x, t), L_k(x)) + (g(x, t), L_k(x)), \quad \forall \quad L_k(x)\in \P_N(\Omega).
\end{equation}
 Now
 \begin{eqnarray*}
 (u^N_{tt}(x, t), L_k(x)) &=& \sum_{k=0}^{N} \frac{d^2}{dt^2} a_j(t) (L_j(x), L_k(x)) \\
 &=& \frac{2}{2k+1} \frac{d^2}{dt^2} a_k(t),
 \end{eqnarray*}
and
\[
 (g(x, t), L_k(x)) =  \sum_{l=0}^{N} w_l g(x_l) L_k(x_l).
\]
where $w_l$ are Gauss weights and $x_l$ are Gauss quadrature points.
We define
\begin{eqnarray*}
\mathcal{L} u(x, t) &=&  \mathcal{L}_1 u(x, t) + \mathcal{L}_2 u(x, t) \\
              &=& \int \mj(x-y) u(y, t) dy - u(x, t) \int \mj(x-y)  dy.
\end{eqnarray*}
Now
\[
(\mathcal{L} u^N (x, t), L_k(x)) = (\mathcal{L}_1 u^N (x, t), L_k(x)) + (\mathcal{L}_2 u^N (x, t), L_k(x))
\]
where
\[
 (\mathcal{L}_1 u^N (x, t), L_k(x)) = \sum_j a_j(t)\int \int \mj(x-y) L_k(x) L_j(y) dydx
\]
and
\[
 (\mathcal{L}_2 u^N (x, t), L_k(x)) = \sum_j a_j(t)\int L_k(x) L_j(x)\int \mj(x-y)dy  dx.
\]
Since we consider $\mj$ such that
\[
\int_{\Omega} \mj(x) =1, \quad \forall \quad x\in\Omega,
\]
we have
\[
 (\mathcal{L}_2 u^N (x, t), L_k(x)) = \sum_j a_j(t)\int L_k(x) L_j(x)  dx=\frac{2}{2k+1} a_k(t).
\]
Adding all these integrals (\ref{f:specweak02}) can be written as
\begin{eqnarray}\label{f:specweak02aa}
 \frac{2}{2k+1} \frac{d^2}{dt^2}a_k(t) &=& \sum_{j=0}^N a_j(t) \int_{\Omega} \int_{\Omega} \mj(x-y) L_k(x) L_j(y) dy dx \nonumber\\
 &&\quad - \frac{2}{2k+1} a_k(t) + \int g(x, t) L_k(x) dx,
\end{eqnarray}
for all $k=0, 1, \cdots, N$. Thus \eqref{f:specweak02aa} can be written as
\[
   M \frac{d^2}{dt^2}a(t) = A a(t) + b(t),
\]
with $a(0)=M^{-1}u_0$, and $a'(0) = M^{-1}v_0$,
$A=A_1+A_2$ and elements of $A_1$ are defined from the continuous operator $\mathcal{L}_1$ as well as the elements of  $A_2$ are defined from the continuous operator of $\mathcal{L}_2$.
Since $M$ is a nonsingular diagonal matrix, the above second order ordinary differential equation can be written as
\begin{equation}\label{eq:0020}
\frac{d^2}{dt^2}a(t) = M^{-1}A a(t) + M^{-1}b(t).
\end{equation}
%
 We approximate the integrals by Gaussian quadrature formula
\begin{align*}
 \int_{\Omega}\int_{\Omega}  \mj (x-y) L_{k}(x) L_j(y) dy dx
& \approx \sum_{l=0}^{M}w_lL_{k}(x_l)\int_{\Omega}  \mj (x_l-y)  L_j(y) dy\\
& \approx \sum_{l=0}^{M}\sum_{m=0}^{M} w_l L_{k}(x_l) w_m  \mj (x_l-x_m)  L_j(x_m),
\end{align*}
and
\begin{align*}
  \int_{\Omega} g(x, t) L_{k}(x) dx &\approx \sum_{l=0}^{M} w_l g(x_l, t) L_j(x_l),
\end{align*}
where $x_l$ are quadrature points, and $w_l$ are quadrature weights. \eqref{f:specweak02aa} is a system of second order time dependent ordinary differential equations which can be solved using a simple finite difference scheme to approximate the coefficients $a_k(t)$.

Depending on the nature of the kernel function, it is also wise to divide the integral domain into pieces. Then one can use small number of quadrature points to compute the integrals, since over each subinterval/subdomain $J(x)$ will be flatter (if $J(x)$ is a smooth function over $\Omega$.) We discuss the quadrature  approximations in Section~\ref{f:section001}.
\section{Numerical experiments and discussion}
In this section, we exhibit several numerical results to illustrate the scheme. Here we  use a simple scheme for time integration. We use MATLAB to serve our purpose. We approximate the time dependent system of equations \eqref{eq:0020} using a central difference scheme~\footnote{Mathematica and MATLAB builtin functions can also be used to solve the time dependent system of equations \eqref{eq:0020}.}. For simplicity we consider $\Omega=[-1, 1]$ and the scheme can be extended to $[-L, L]$, for any $L>0$. We approximate \eqref{eq:0020} by
\begin{equation}\label{eq:0021}
\frac{a^{j+1}-2a^{j}+a^{j-1}}{dt^2} = M^{-1}A a^{j+1} + M^{-1}b(t_j)
\end{equation}
where $a^j=a(t_j).$
 This scheme is accurate~\cite{Atkinson} of order $\mathcal{O}(\dt^2 +N^{-m})$.

 We know that $\rho (A)\le \|A\|$ for any matrix $A$ where $\rho(A)$ stands for condition number of $A$. Also
$\|\frac{1}{I-\dt^2 M^{-1}A}\| \le \frac{1}{1-\|dt^2 M^{-1}A\|}.$  Now
\begin{eqnarray*}
|A_{jk}^1|=| \int \int \mj(x-y) L_k(x) L_j(y) dy dx|
    &\le \max_{x\in \Omega} \mj(x)|\int \int L_k(x) L_j(y) dy dx|\\
    &\le | \int  L_k(x) dx| |\int  L_j(y) dy| \le 4 K,
\end{eqnarray*}
since
\[
\int_{\Omega} L_j(x) dx =\left\{
                         \begin{array}{cc}
                         2 & j=0\\
                         0 & j\ne 0,
                         \end{array}
                         \right.,
\]
 $\mj$ is a normalized function and $K=\max_{x\in \Omega} |\mj (x)|$. We have
\[
\|A_1\|=\sqrt{\sum_{i,j=1}^{N+1}|A_{ij}^1|^2} \le \sqrt{\sum_{i,j=1}^{N+1} 16K}=\sqrt{16(N+1)^2K^2}=4(N+1)K,
\]
and
\[
\|A_2\| = \sqrt{\sum_{j,j=1}^{N+1} |A^{2}_{jj}|^2}\le 2\sqrt{\sum_{j=1}^{N+1} \left|\frac{1}{2j-1}\right|^2 }< 4.
\]
Thus
\[
\|A\| \le \|A_1\| + \|A_2\| \le 4\left((N+1)K+1\right),
\]
and
$
1\le\|M\| \le 4
$,
shows that $\|M^{-1}A\| \le 4\left((N+1)K+1\right)$ which is a bounded quantity for any finite $N$.
So $\dt^2\|M^{-1}A\|\longrightarrow 0$ as $\dt \longrightarrow 0$.
Hence
\[
\rho\left(\frac{1}{I-\dt^2 M^{-1}A}\right)\le \|\frac{1}{I-\dt^2 M^{-1}A}\|
\le \frac{1}{1-\|\dt^2 M^{-1}A\|}
\le 1,
\]
and thus by the definition of condition number of a matrix
\[
 \rho\left(\frac{1}{I-\dt^2 M^{-1}A}\right)=1,\ \text{as}\ \dt\rightarrow 0\ \text{and}\ N<\infty.
\]
Thus the scheme \eqref{eq:0021} is unconditionally stable.

In the following Figures~\ref{fig:slutns01}-\ref{fig:slutns02}, we present the numerical solutions for various choices of initial functions, $\rho$, and $g(x, t)$ to demonstrate the scheme considering $\mj(x)=\sqrt{\frac{400}{\pi}}e^{-400 x^2}$. Here we notice that the wave form is similar to the solutions of local linear wave model and the numerical solutions behaves well for any finite time preserving wave phenomena.
\begin{figure}[here]
\begin{center}
\includegraphics[width=0.49\textwidth,height=6.5cm]{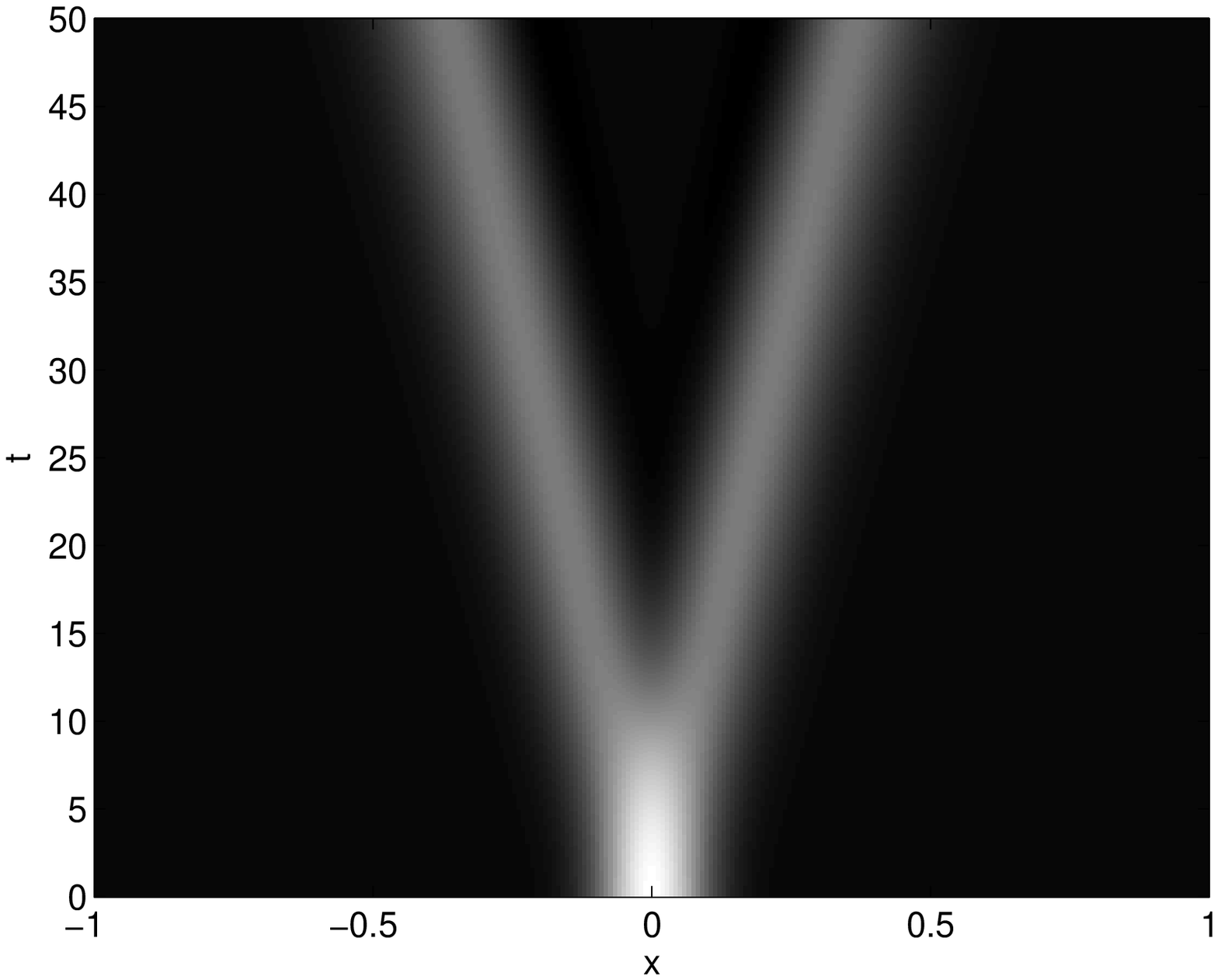}
\includegraphics[width=0.49\textwidth,height=6.5cm]{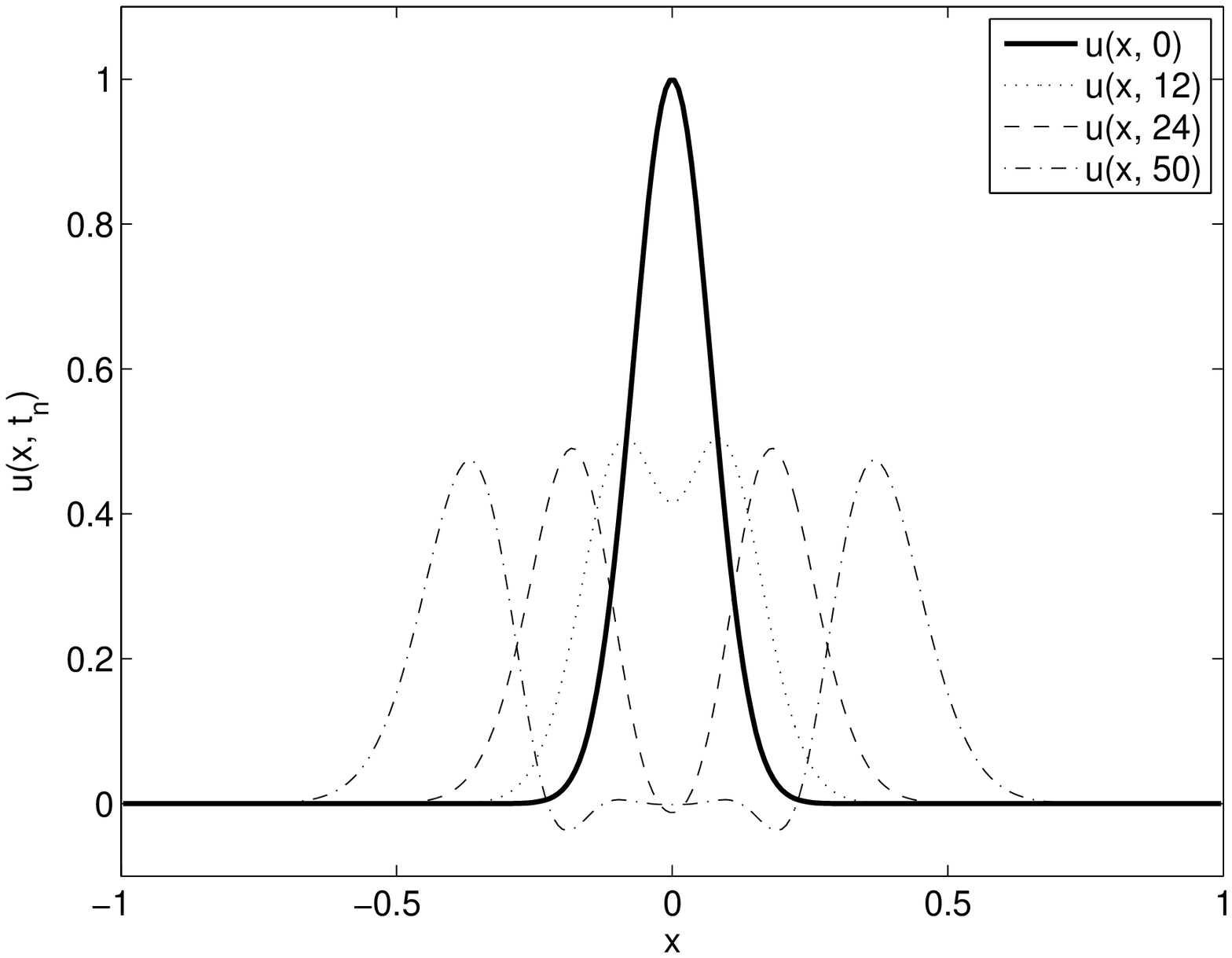}
\includegraphics[width=0.49\textwidth,height=6.5cm]{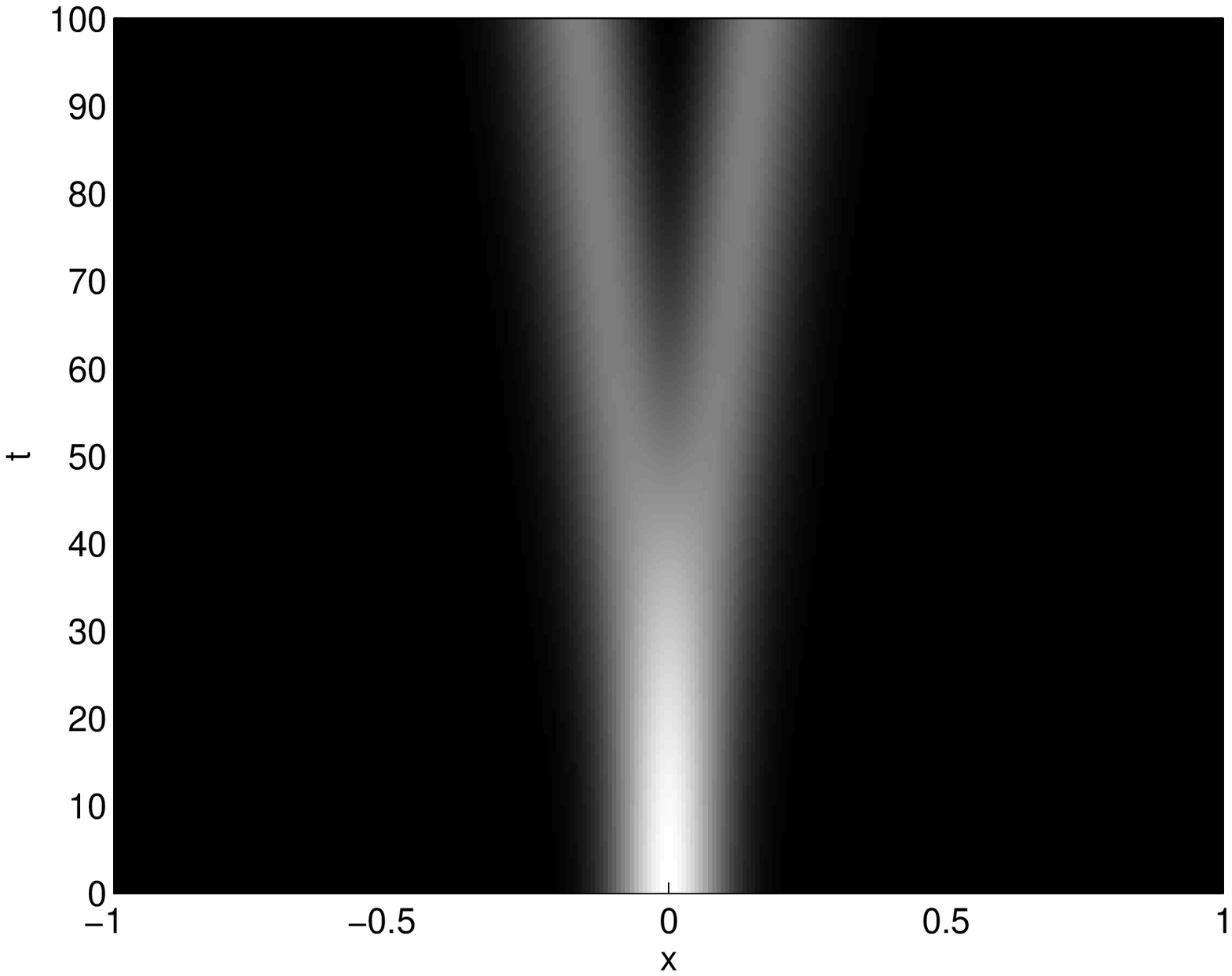}
\includegraphics[width=0.49\textwidth,height=6.5cm]{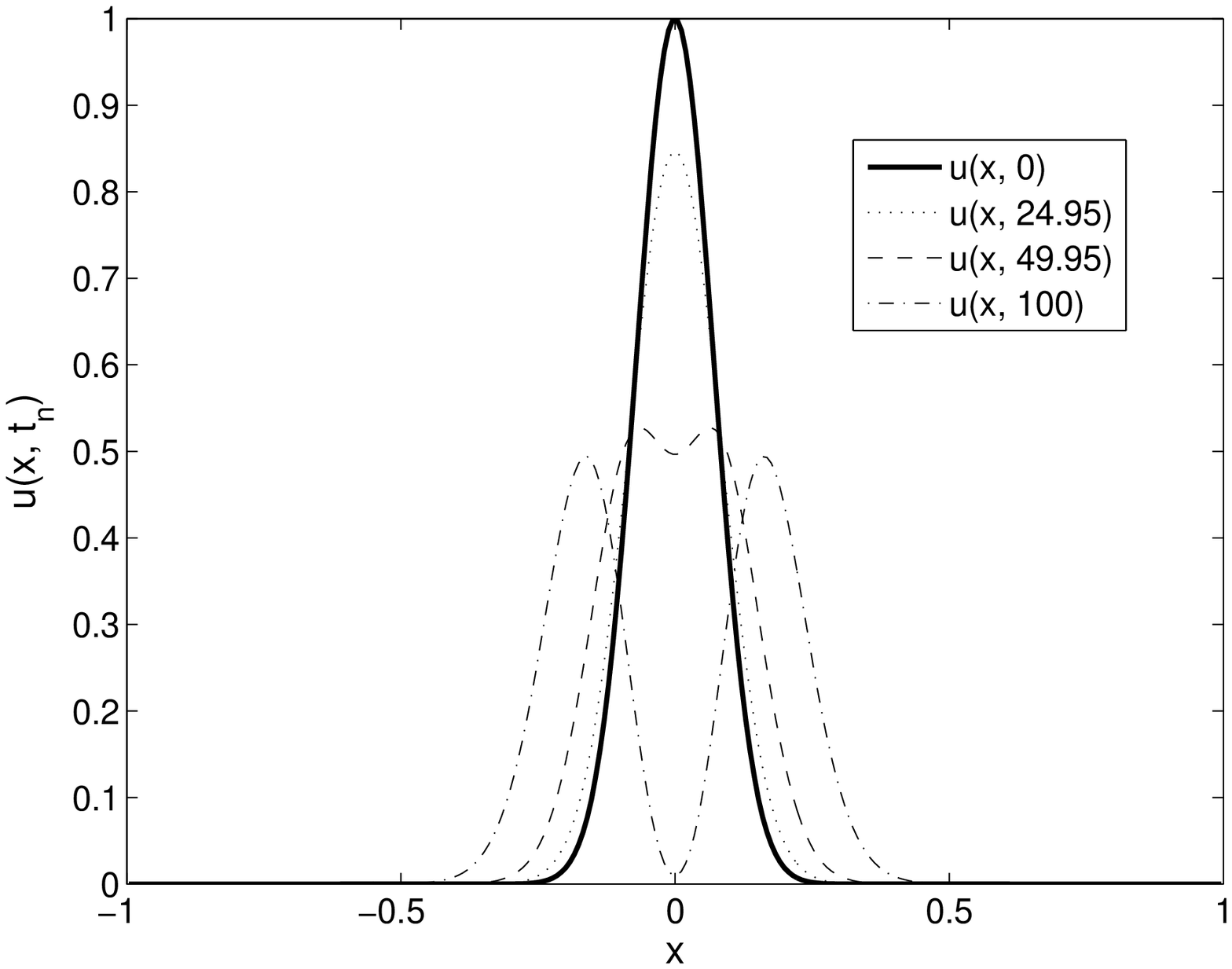}
\caption{ We consider $u(x, 0)= \exp(-100 x^2)$, $u_t(x, 0)=0$, $\dt=0.05$,  $\rho=.1$(upper figures) and $\rho=0.01$ (bottom figures).}
\label{fig:slutns01}
\end{center}
\end{figure}
\begin{figure}[here]
\begin{center}
\includegraphics[width=0.49\textwidth,height=6.5cm]{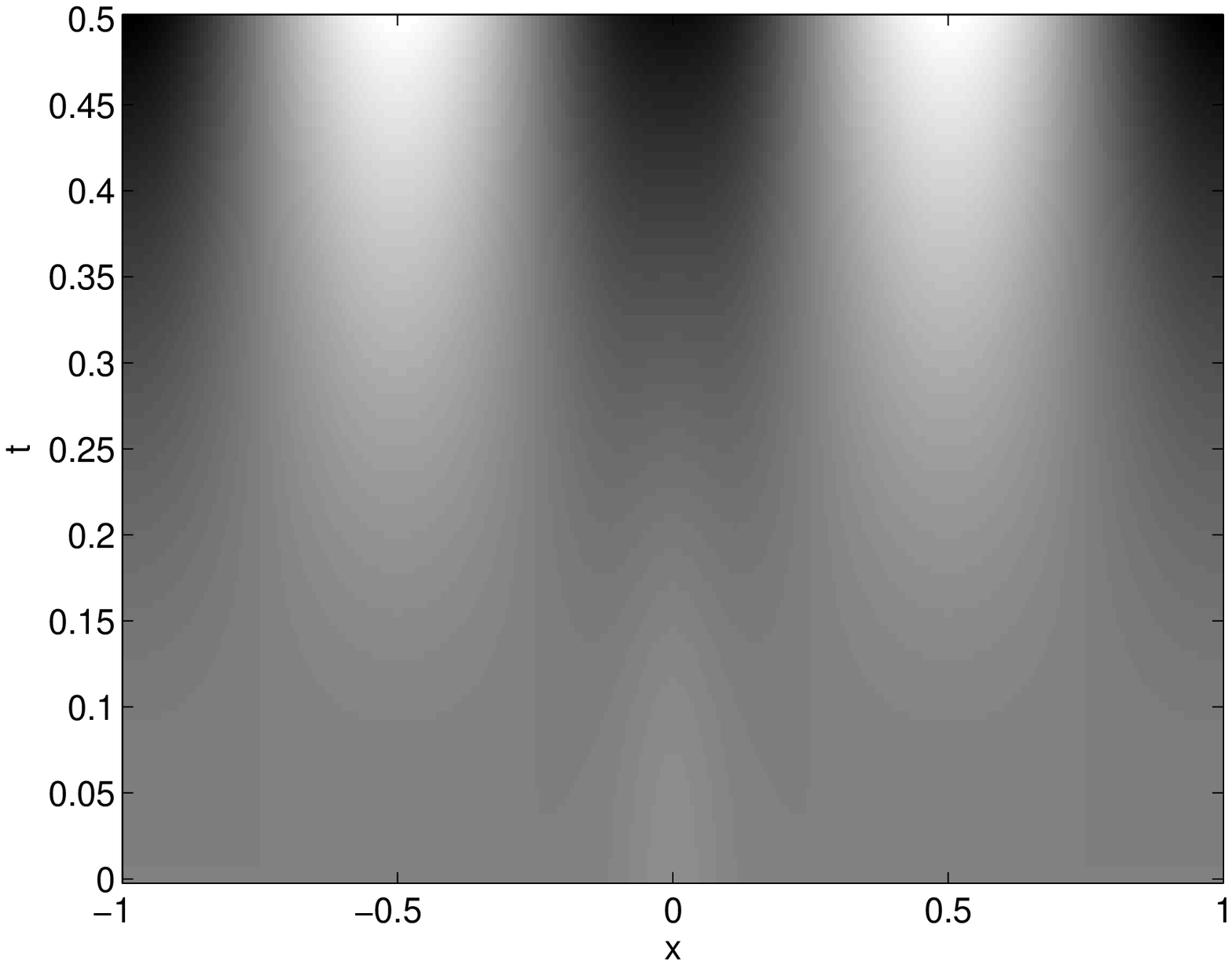}
\includegraphics[width=0.49\textwidth,height=6.5cm]{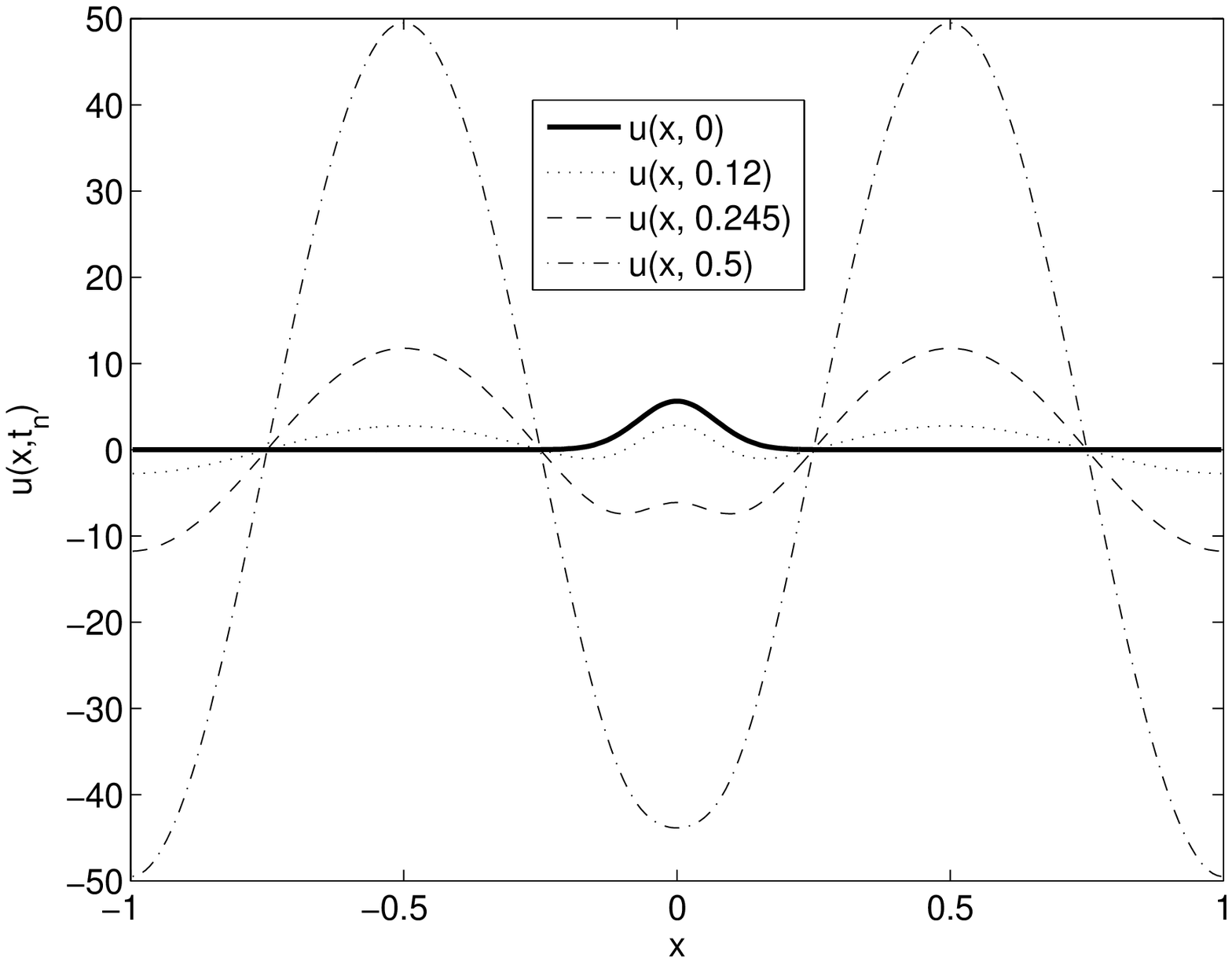}
\caption{ We consider $u(x, 0)= \sqrt{100/\pi}\exp(-100 x^2)$, $u_t(x, 0)=0$, $\rho=.01$, $\dt=0.005$, and $g(x)=-10^{-2} \cos(2\pi x)$.}
\label{fig:slutns02}
\end{center}
\end{figure}
Here in this study we have shown theoretically that the error in spectral approximation decays exponentially. More precisely, it is shown that if the kernel function and the solution underlying the model problem is smooth, then the errors obtained by the Legendre spectral scheme decays exponentially, which one desires from spectral method.

There is a drawback of this approximation. Because of the presence of the convolution operator, the discrete operator is a dense matrix and as a result numerical computations become  slower, needs huge storage for computations in multidimensional domain.

\appendix
\section{Quadrature Approximation}\label{f:section001}
We define $x^k$, $k=1, 2, \cdots, M$ as quadrature points.
Now in each point $x^k$, we write (\ref{eq:001}) as
\begin{eqnarray*}
\frac{\partial^2}{\partial t^2} u(x^k,t) &=& \rho \int_{\Omega} \mj(x^k-y)(u(y, t)-u(x^k, t)) dy + g(x^k,t)
\end{eqnarray*}
which can be approximated by
\begin{eqnarray}\label{eq:00001a}
\frac{\partial^2}{\partial t^2} u(x^k,t) 
&=&\rho \sum_{i=1}^M w^i \mj(x^k-x^i)(u(x^i, t)-u(x^k, t)) + g(x^k,t),
\end{eqnarray}
where $w^i$ are quadrature weights. \eqref{eq:00001a} is a system of second order ordinary differential equations, which can be solved by any standard techniques.

One may also use a few quadrature points by subdividing the domain into several subdomains. We subdivide the domain $\Omega$ into $\Omega_j$, $j=1, 2, \cdots, N_h$ subdomains such that $\Omega=\sum_j \Omega_j$. On each of the subdomains $\Omega_j$, $j=1, 2, \cdots, N_h$ we define  $K$ points $x_j^k$, $k=1, 2, \cdots, K$ as quadrature points.

Now in each point $x_j^k$, we write (\ref{eq:001}) as
%
\begin{eqnarray*}
\frac{\partial^2}{\partial t^2} u(x_j^k,t) &=& \rho \int_{\Omega} \mj(x_j^k-y)(u(y, t)-u(x_j^k, t)) dy + g(x_j^k,t)
\\
&=&\rho \sum_{m=1}^{N_h}\int_{\Omega_m} \mj(x_j^k-y)(u(y, t)-u(x_j^k, t)) dy + g(x_j^k,t),
\end{eqnarray*}
which can be approximated by $K$ point Gauss quadrature formula as
\begin{eqnarray}\label{eq:001a}
\frac{\partial^2}{\partial t^2} u(x_j^k,t) 
&=&\rho \sum_{m=1}^{N_h}\sum_{i=1}^K w_m^i \mj(x_j^k-x_m^i)(u(x_m^i, t)-u(x_j^k, t)) + g(x_j^k,t),
\end{eqnarray}
%
where $w_m^i$ are quadrature weights.
\eqref{eq:001a} can be solved using any  linear  ordinary differential equation solvers.

Here, in Figure~\ref{fig:slutns01aa}, we demonstrate a solution of the model in two space dimensions using mid-point quadrature formula for space integration followed by Euler's formula for time integration in a periodic $(0, 1)^2$ space domain.
\begin{figure}[here]
\begin{center}
\includegraphics[width=0.95\textwidth,height=12.5cm]{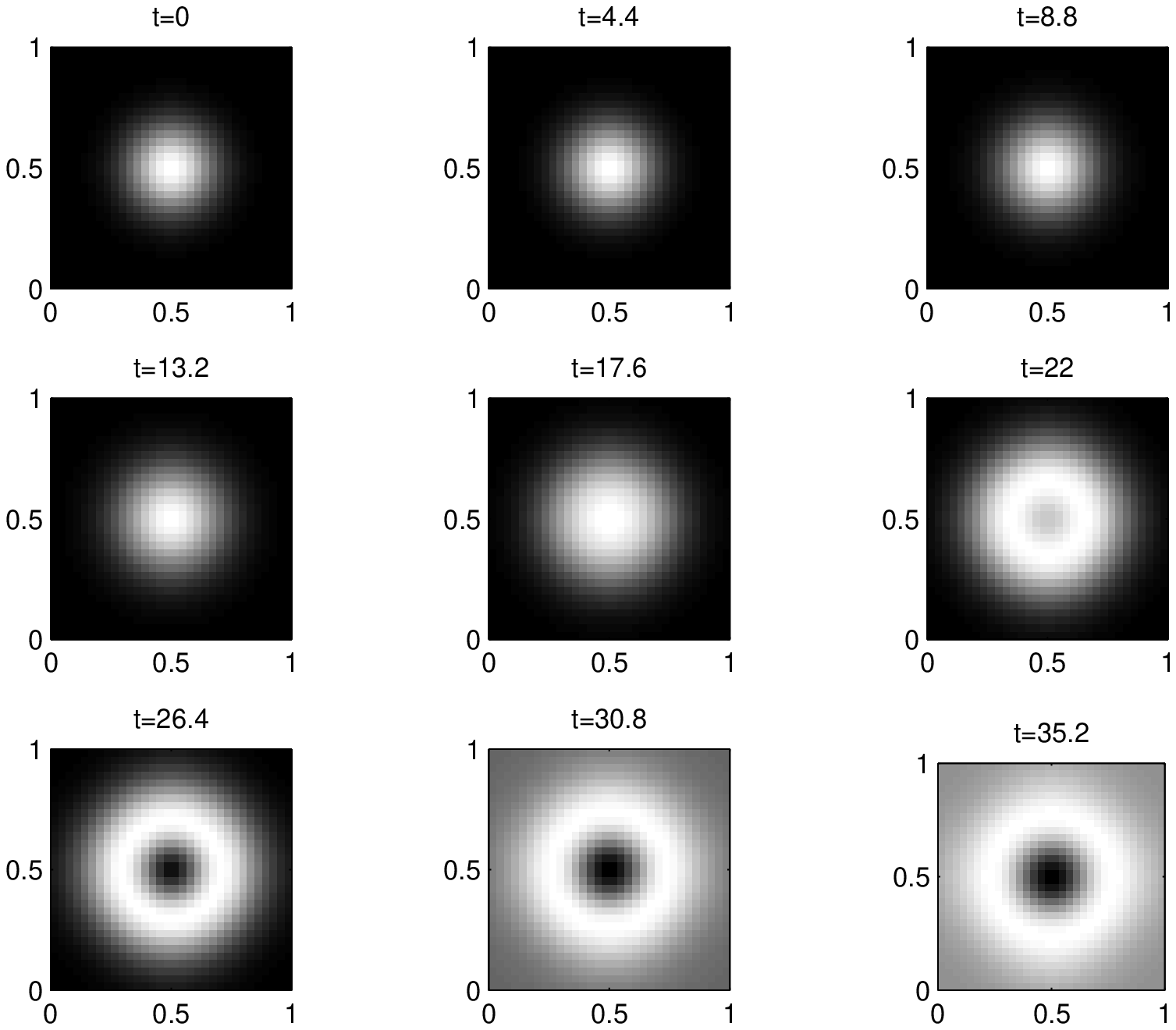}
\caption{ We consider $\rho=0.1$, $dt=0.1$, $N=32$ space elements, $u(x, y, 0)=\exp\left(-10\left((x-.5)^2)+(y-.5)^2\right)\right)$, $u_t(x, y, 0)=0$.}
\label{fig:slutns01aa}
\end{center}
\end{figure}


\end{document}